\documentclass[12pt,a4paper]{amsart}
\usepackage{amssymb}
\usepackage[T1]{fontenc}
\usepackage{Alegreya}
\usepackage[libertine]{newtxmath}

\usepackage{mathrsfs}
\usepackage{tikz-cd}
\usepackage[headings]{fullpage}
\usepackage{paralist}
\usepackage{xspace}

\usepackage{color}
\usepackage[normalem]{ulem}

\usepackage[colorlinks=true,linkcolor=blue, citecolor=blue, urlcolor=blue,%
pagebackref=true]{hyperref}

\theoremstyle{plain} \newtheorem{theorem}[equation]{Theorem}
\newtheorem{lemma}[equation]{Lemma}
\newtheorem{corollary}[equation]{Corollary}
\newtheorem{proposition}[equation]{Proposition}

\theoremstyle{definition}

\newtheorem{remark}[equation]{Remark} \newenvironment{remarkbox}[1][]{%
\begin{remark}[#1] \pushQED{\qed}}{\popQED \end{remark}}

\newtheorem{example}[equation]{Example} \newenvironment{examplebox}[1][]{%
\begin{example}[#1] \pushQED{\qed}}{\popQED \end{example}}

\newtheorem{definition}[equation]{Definition}

\newtheorem{notation}[equation]{Notation}

\newtheorem{discussion}[equation]{Discussion}
\newenvironment{discussionbox}[1][]{%
\begin{discussion}[#1]\pushQED{\qed}}{\popQED \end{discussion}}

\newtheorem{observation}[equation]{Observation}
\newenvironment{observationbox}[1][]{%
\begin{observation}[#1]\pushQED{\qed}}{\popQED \end{observation}}

\newtheorem{construction}[equation]{Construction}

\newtheorem{setup}[equation]{Setup}
 \newcounter{step}

\newcommand{\fraka}{{\mathfrak a}}

\newcommand{\calJ}{\mathcal J}

\newcommand{\frakM}{{\mathfrak M}}
\newcommand{\calM}{\mathcal M}
\newcommand{\frakm}{{\mathfrak m}}

\newcommand{\calN}{\mathcal N}
\newcommand{\frakn}{{\mathfrak n}}

 \let\strSh\calO

\newcommand{\frakP}{{\mathfrak P}}
\newcommand{\frakp}{{\mathfrak p}}

\newcommand{\calR}{\mathcal R}

\DeclareMathOperator{\ann}{Ann}

\newcommand{\naturals}{\mathbb{N}}

\def\to{\longrightarrow}
\newcommand{\Rees}{\mathscr R}
\DeclareMathOperator{\height}{ht}

\DeclareMathOperator{\Hom}{Hom}

\DeclareMathOperator{\Spec}{Spec}
\DeclareMathOperator{\Proj}{Proj}
\DeclareMathOperator{\Min}{Min}

\DeclareMathOperator{\homology}{H}
\newcommand{\define}[1]{\emph{#1}}
\newcommand{\minus}{\ensuremath{\smallsetminus}}
\DeclareMathOperator{\image}{Im}
\makeatletter
\def\RDerChar{\mathbf{R}}
\def\RDer{\@ifnextchar[{\R@Der}{\ensuremath{\RDerChar}}}
\def\R@Der[#1]{\ensuremath{\RDerChar^{#1}}}

\makeatother

\newif\ifreadkumminibib
\readkumminibibfalse

\begin{document}
\title{The non-$F$-rational locus of Rees algebras}
\author{Nirmal Kotal}
\address{Chennai Mathematical Institute, Siruseri, Tamilnadu 603103. India}
\email{nirmal@cmi.ac.in}
\author{Manoj Kummini}
\address{Chennai Mathematical Institute, Siruseri, Tamilnadu 603103. India}
\email{mkummini@cmi.ac.in}
\thanks{Both authors
were partly supported by an Infosys Foundation fellowship.}


\begin{abstract}
In this note, we give a description of the parameter test submodule of Rees
algebras. This, in turn, describes the non-$F$-rational locus.
\end{abstract}

\maketitle

\section{Introduction}
\numberwithin{equation}{section}

Let $(R, \frakm)$ be a $d$-dimensional excellent local
domain of prime characteristic $p>0$, where $d \geq 2$.
Let $I$ be an $\frakm$-primary ideal.
Write $\Rees_R(I)$ for the Rees algebra $\oplus_{n \in \naturals}I^nt^n$
where $t$ is a variable of degree $1$.
For a domain $A$, we write $\overline{A}$ for its normalization.
In this paper, we prove the following theorem:

\begin{theorem}
\label{theorem:test_submodule_charp}
Let $(R,\mathfrak{m})$ be a Cohen-Macaulay complete normal local domain of
dimension at least $2$ and of characteristic $p > 0$.
Let $I$ be an $R$-ideal that has a reduction generated by a system of
parameters.
Write $\Rees = \Rees_R(I)$.
Suppose that $\overline{\Rees}$ is Cohen-Macaulay.
Then the parameter test submodule
$\tau(\omega_{\overline{\Rees}})$ equals
$\bigoplus_{n\geq 1} \tau (\omega_R, I^n)$.
Consequently,
the non-$F$-rational locus of $\overline\Rees$ is the support of the
$\overline\Rees$-module 
\[
\omega_{\overline{\Rees}} / \oplus_{n \geq 1}\tau(\omega_R,I^n).
\]
\end{theorem}

A \define{reduction} of $I$ is an $R$-ideal $J \subseteq
I$ such that $\Rees(I)$ is a finite algebra over the subring
$\Rees(J)$. If $R/\frakm$ is infinite, then
$I$ has a reduction generated by a system of parameters.
Note also that $\overline{\Rees(I)}$ is a finite $\Rees(I)$-algebra since
$R$ is excellent.

\begin{corollary}
\label{corollary:dim2}
Let $(R, \frakm)$ be a two-dimensional $F$-finite Gorenstein complete local
domain with an infinite residue field.
Let $I$ be an $\frakm$-primary ideal such that
$\overline{G} := 
\oplus_{n \in \naturals } \overline{I^n}/\overline{I^{n+1}}$ 
(i.e.  the associated graded ring for the integral closure filtration)
is Gorenstein with $a$-invariant $a(\overline{G}) = -1$.
Suppose that $R$ or $\Proj \overline{\Rees}$ is $F$-rational.
Then $\overline{\Rees}$ is $F$-rational.
\end{corollary}

Our motivation for Theorem~\ref{theorem:test_submodule_charp} is the
following result in characteristic
zero. While we believe that this might have been known, we could not find a
proof; hence we have included a proof in Section~\ref{section:charzero}.
We denote by $\calJ(\omega_R, I^n)$ the multiplier submodule of $I^n$; this
will be defined in Section~\ref{section:charzero}.

\begin{theorem}
\label{theorem:charzero}
Let $(R,\mathfrak{m})$ be a Cohen-Macaulay complete normal local domain of
dimension
$d\geq 2$ and essentially of finite type over a field of characteristic
zero.
Let $I$ be an $R$-ideal such that $\overline{\Rees(I)}$ is
Cohen-Macaulay.
Then the irrational locus of $\overline{\Rees(I)}$
is the support of the $\overline{\Rees(I)}$-module
\[
\omega_{\overline{\Rees}} / \oplus_{n \geq 1}\calJ(\omega_R,I^n).
\]
\end{theorem}

In Section~\ref{section:prelims} we recall the relevant definitions
and results.
Section~\ref{section:pfThmTestIdeal} contains the proofs of
Theorem~\ref{theorem:test_submodule_charp}, 
Corollary~\ref{corollary:dim2}
and an example.
As mentioned above, we sketch a proof of 
Theorem~\ref{theorem:charzero} in Section~\ref{section:charzero}.

\subsection*{Acknowledgements}
We thank the referee and Karl Schwede for helpful comments.

\section{Preliminaries}
\label{section:prelims}
\numberwithin{equation}{subsection}

All the rings we consider in this article are excellent. The letter $p$
denotes a prime number. When used in the context of the Frobenius map and
singularities in prime characteristic, $q$ denotes an arbitrary power of
$p$. Since we have to discuss a local ring and a positively graded algebra
over it, we will switch to the language of *local rings, after we introduce
the basic definitions in tight closure.
See~\cite[\S 1.5 and \S 3.6]{BrHe:CM} for definitions.
Every local ring is *local, by thinking of it as concentrated in
degree $0$. Similarly, modules over local rings can be thought of as graded
modules concentrated in degree $0$.

\subsection*{Tight closure}

Let $R$ be a ring of prime characteristic $p>0$.
Let $F:R\to R$ be the \emph{Frobenius map} $r \mapsto r^p$.
For $e\in \mathbb{N}$, we denote by ${}^e\!R$ the $R$-module $R$ via the
$e$-fold iterated Frobenius map $F^e: R\to R$.
For an $R$-ideal $I$ and $e \in \naturals$, $I^{[p^e]}$ is
the $R$-ideal generated by $\{r^{p^e} \mid r \in I \}$.
Let $M$ be an $R$-module and write $F^e(M) = {}^e\!R\otimes_RM$.
For $x \in M$, its image in $F^e(M)$ under the
natural map $M\to {{}^e\!R}\otimes_R M$ is written as $x^{p^e}$.
For a submodule $N\subseteq M$, the image of $F^e(N)$
in $F^e(M)$ is denoted by $N^{[p^e]}_M$.
Write
\[
R^{0} = R \minus \bigcup_{\frakp \in \Min R }\frakp.
\]

Let $\fraka$ be an $R$-ideal and $N \subseteq M$ be $R$-modules.
The \define{$\mathfrak{a}$-tight closure} of $N$ in 
$M$~\cite[Definition~1.1]{HaraYoshidaGenTightClMultIdeals2003}
is
\[
N^{*\mathfrak{a}}_M :=
\{z\in M \mid
\text{there exists}\;
c\in R^{0} \;\text{such that}\; c\mathfrak{a}^qz^q\subseteq N^{[q]}_M
\;\text{for all}\; q\gg 1 \}.
\]
When $\fraka = R$, this is the same as the tight closure defined
in~\cite{HochsterHunekeTCInvThyBSThm90}; we then
write $N^{*}_M$ for $N^{*R}_M$.
When $M=R$ and $N=I$ an $R$-ideal, we
write $I^{*\mathfrak{a}}$ and $I^*$ respectively.

A \define{parameter ideal} of $R$ is an $R$-ideal $I$ that is minimally
generated by $\height I$ elements.
We say that $R$ is \emph{$F$-rational} if $I=I^*$ for every parameter ideal
$I$.
A \define{parameter test element} is an element of $\cap (I :_R I^*)$ where
the intersection runs over all parameter ideals $I$.

\subsection*{Matlis duality}
Let $\calR = \oplus_{i \geq 0 } \calR_i$ be a noetherian *local ring, with
*maximal ideal $\frakM$ and
$\mathcal{M}$ a graded $\mathcal{R}$ module. 
Let $E$ be the injective hull of $\calR_0/(\calR_0 \cap \frakM)$ as an
$\calR_0$-module.
The \emph{(graded) Matlis dual} of $\mathcal{M}$
is $\mathcal{M}^{\vee} := \bigoplus_{i\in \mathbb{Z}}
\left[\mathcal{M}^{\vee} \right]_i$, where
\[
\left[\mathcal{M}^{\vee} \right]_i := 
\Hom_{\mathcal{R}_0} \left( \mathcal{M}_{-i},E \right).
\]
If $\calN \subseteq \calM^\vee$ then we think of $(\calM^\vee/\calN)^\vee$
as $\ann_{\calM}(\calN)$ for the following reason:
For each $f \in \calM \subseteq {\calM^\vee}^\vee$, 
\[
f(\calN) = 0 \;\text{if and only if}\;
f \in \Hom_R(\calM^\vee/\calN,E) = (\calM^\vee/\calN)^\vee.
\]

Now assume that $\calR$ is reduced and *complete. 
If $\mathcal{M}$ is either a noetherian or
an artinian $\mathcal{R}$-module, then
$\left(\mathcal{M}^{\vee}\right)^{\vee}=\mathcal{M}$.
Define the \define{(graded) canonical module}
\[
\omega_{\calR} = \homology^{\dim \calR}_\frakM(\calR)^\vee.
\]

Further assume that $\calR$ has prime characteristic $p>0$.
Let $\fraka$ be a graded $\calR$-ideal not contained in any minimal prime
ideal of $\calR$.
The \define{parameter test submodule}  $\tau(\omega_\calR,\fraka)$
associated to the pair $(\calR,\fraka)$ is
\[
\tau(\omega_\calR,\fraka) := 
\ann_{\omega_\calR}\left(0^{*\fraka}_{\homology^{\dim \calR}_\mathfrak{M}
(\calR)} \right).
\]
When $\fraka = \calR$, we write $\tau(\omega_\calR)$ instead of 
$\tau(\omega_\calR,\calR)$. 
When $\calR$ is Gorenstein, this agrees with the definition of test ideal
$\tau(\fraka)$ in~\cite{HaraYoshidaGenTightClMultIdeals2003}.
If $\calR$ is reduced and Cohen-Macaulay, 
for a prime ideal $\frakp \in \Spec \calR$,
$\calR_\frakp$ is $F$-rational if and only if 
$\left( \omega_\calR/\tau(\omega_\calR)\right)_\frakp = 0$.
See \cite{smith95testideals} and \cite[Remark 6.4]{st08rationalsing} for
details.

\section{Proof of Theorem~\protect{\ref{theorem:test_submodule_charp}}}
\label{section:pfThmTestIdeal}
\numberwithin{equation}{section}

Since we are concerned about $\overline{R[It]}$
in Theorem~\ref{theorem:test_submodule_charp}, we may replace 
(using~\cite[Proposition~1.3]{HaraYoshidaGenTightClMultIdeals2003})
$I$ by a minimal reduction, which, by hypothesis, is generated by a system
of parameters.

\begin{setup}
\label{setup:main}
Let $(R,\mathfrak{m})$ be an excellent Cohen-Macaulay normal local domain
of dimension $d\geq 2$ and of characteristic $p>0$.
Let $I=(f_1,\cdots,f_d)$ be a parameter ideal, that is $f_1,\cdots ,f_d$ is
a system of parameters (and hence a regular sequence).
Let $f=f_1\cdots f_d$.
Set $I^{[l]}:=(f_1^l,\cdots,f_d^l)$ for $l\geq 1$.
Denote the Rees algebra $R[It]$ (respectively $\overline{R[It]}$) by
$\Rees$ (respectively $\overline{\Rees}$).
Write $\mathfrak{M}$ for the *maximal ideal of $\Rees$.
Assume that $\overline{\Rees}$ is Cohen-Macaulay.
\end{setup}

\begin{discussionbox}
\label{discussionbox:jntightclosure}
Assume Setup~\ref{setup:main}.
Let $a\in R$ and $l, n$ be positive integers.
Then $a\in I^{[l]^{*I^n}}$ if and only if
$\left[\frac{a}{f^l} \right]\in 0^{*I^n}_{\homology^d_\frakm(R)}$.
We prove this as follows:
$a\in I^{[l]^{*I^n}}$ if and only if 
there exists a nonzero $c \in R$ 
such that $c a^q I^{nq}\subseteq I^{[lq]}$ for all $q\gg 1$.
Now, 
$c a^q I^{nq}\subseteq I^{[lq]}$ if and only if
$c I^{nq}\left[\frac{a^q}{f^{lq}}\right]=0 \in \homology^d_\frakm(R)$;
see~\cite[Proof of Theorem 2.1, p.~104--105]{LipTeiPseudoRatlSing81}.
Now apply the definition of $0^{*I^n}_{\homology^d_\mathfrak{m} (R)}$.
\end{discussionbox}

\begin{discussionbox}
\label{discussionbox:descrLC}
We quote some observations from \cite[Section 3]{kotal2023blowup}. 
These results were proved
in~\cite{HaraWatanabeYoshidaFrationality2002,HaraYoshidaGenTightClMultIdeals2003}
for the $I$-adic filtration and later generalized for the integral closure
filtration in~\cite{kotal2023blowup}.
For each $n \geq 1$ we have an exact sequence
\begin{equation}
\label{equation:locohExactSeq}
\begin{tikzcd}
0 \arrow[r] &
\left[ \homology^d_{\overline{\Rees}_{+}} (\overline{\Rees})
\right]_{-n} \arrow[r, "\phi_{-n}"]
&
\homology^d_\mathfrak{m} (R)t^{-n} \arrow[r,"\psi_{-n}"]
&
\left[ \homology^{d+1}_{\mathfrak{M}} (\overline{\Rees}) \right]_{-n}
\arrow[r]
& 0 .
\end{tikzcd}
\end{equation}
Moreover, for all $a \in R$ and all positive integers $l,n$ such that 
$dl > n$,
we have that
$\left[\frac{a}{f^l}\right]t^{-n}\in \image (\phi_{-n})$
if and only if $a \in
\overline{I^{dl-n}}+I^{[l]}$~\cite[Proposition~3.9]{kotal2023blowup}.
Let $c\in R$.
Then for all $n\geq 1$, we have,
the following commutative diagram~\cite[Discussion~3.12]{kotal2023blowup}:
\begin{equation}
\label{equation:locohExactSeqFrob}
\begin{tikzcd}
0 \arrow[r] & \left[ \homology^d_{\overline{\Rees}_{+}} (\overline{\Rees})
\right]_{-n} \arrow[r, "\phi_{-n}"] \arrow[d,"cF^e"]
& \homology^d_\mathfrak{m} (R)t^{-n} \arrow[r, "\psi_{-n}"] \arrow[d, "cF^e"]
&\left[ \homology^{d+1}_{\mathfrak{M}} (\overline{\Rees}) \right]_{-n}
\arrow[r] \arrow[d, "cF^e"]  & 0
\\
0 \arrow[r] & \left[ \homology^d_{\overline{\Rees}_{+}} (\overline{\Rees})
\right]_{-nq} \arrow[r, "\phi_{-nq}"]
& \homology^d_{\mathfrak{m}} (R)t^{-nq} \arrow[r, "\psi_{-nq}"] & \left[
\homology^{d+1}_{\mathfrak{m}} (\overline{\Rees}) \right]_{-nq} \arrow[r]
& 0 .
\end{tikzcd}
\end{equation}
(In the above diagram, $q=p^e$.)
\end{discussionbox}

\begin{remarkbox}
\label{remarkbox:jlstar}
Assume Setup~\ref{setup:main}. Let $l,n$ be integers such that $dl>n$.
Then
\[
I^{[l]^{*I^n}} =\{z\in R: \;\text{there exists}\; c\in R \minus \{0 \}
\;\text{such that}\; cz^q\in \overline{I^{(dl-n)q}}+I^{[lq]} \;\text{for
all}\; q\gg 1\}.
\]
(See~\cite[(3.5)]{kotal2023blowup}.)
\end{remarkbox}

\begin{lemma}
\label{lemma:homomorpism_on_local_cohom}
Assume Setup~\ref{setup:main}.
Let $n \geq 1$. 
Then 
\[
0^{*I^n}_{\homology^d_\mathfrak{m} (R)} t^{-n}
=
(\psi_{-n})^{-1}\left(\left[
0^{*}_{\homology^{d+1}_{\mathfrak{M}} (\overline{\Rees})} 
\right]_{-n}\right).
\]
\end{lemma}

\begin{proof}
We first show that 
\begin{equation}
\label{equation:homomorpism_on_local_cohom:subseteq}
\psi_{-n}\left(0^{*I^n}_{\homology^d_\mathfrak{m} (R)} t^{-n}\right) 
\subseteq \left[0^{*}_{\homology^{d+1}_{\mathfrak{M}}
    (\overline{\Rees})} \right]_{-n}
\end{equation}
Let $\xi = \left[\frac{a}{f^l}\right] \in 
0^{*I^n}_{\homology^d_{\frakm} (R)}$.
Since $\left[\frac{a}{f^l}\right] = \left[\frac{af^{l'}}{f^{l+l'}}\right]$
for every positive integer $l'$, we may assume, 
without loss of generality, that $dl>n$. 
By Discussion~\ref{discussionbox:jntightclosure}
and Remark~\ref{remarkbox:jlstar},
there exists a nonzero $c \in R$ such
that $ca^q\in \overline{I^{(dl-n)q}}+I^{[lq]}$ for all $q\gg 1$. 
Then, by Discussion~\ref{discussionbox:descrLC},
$\left[\frac{ca^q}{f^{lq}}\right]t^{-nq} \in \image (\phi_{-nq})$ for
each $q\gg 1$. 
In other words, $cF^e(\xi t^{-n})\in \image (\phi_{-np^e})$ for each
$e\gg 0$. 
Therefore, by~\eqref{equation:locohExactSeqFrob}, 
$cF^e(\psi_{-n}(\xi t^{-n}))=0$ for each $e\gg 0$, 
which proves~\eqref{equation:homomorpism_on_local_cohom:subseteq}.

Conversely, let 
$\eta \in \left[0^{*}_{\homology^{d+1}_{\mathfrak{M}}
(\overline{\Rees})} \right]_{-n}$.
Let
$\xi = \left[\frac{a}{f^l}\right] \in \homology^d_{\mathfrak{m}}(R)$ 
be such that $\psi_{-n}(\xi t^{-n}) = \eta$.
We want to show that 
$\xi \in 0^{*I^n}_{\homology^d_\mathfrak{m} (R)}$. 
There exists $c \in R$ that is a parameter test element for $R$ and
$\Rees$.
Since $cF^e(\psi_{-n}(\xi t^{-n})) = 0$ for all $e\gg 1$, 
it follows from \eqref{equation:locohExactSeqFrob}
that $cF^e(\xi t^{-n})\in \image \phi_{-np^e}$ for all $e\gg 1$. 
Use Discussion~\ref{discussionbox:descrLC}, Remark~\ref{remarkbox:jlstar}
and Discussion~\ref{discussionbox:jntightclosure}
to conclude that
$\xi \in 0^{*I^n}_{\homology^d_\mathfrak{m} (R)}$.
\end{proof}

\begin{remarkbox}
The above proposition shows that $\image \left ( \phi_{-n}\right )
\subseteq 
0^{*I^n}_{\homology^d_\mathfrak{m} (R)} t^{-n}$ for all $n \geq 1$.
One can see this more directly as follows:
Let $a \in R$ and $l, n$ be positive integers such that 
\[
\left[\frac{a}{f^l}\right] \in  \image \left ( \phi_{-n}\right )
\]
and (without loss of generality) $dl > n$.
Then we have the following sequence of implications.
\begin{align*}
a & \in \overline{I^{dl-n}} + I^{[l]}
\qquad (\text{Discussion~\ref{discussionbox:descrLC}});
\\
a & \in I^{[l]^{*I^n}}
\qquad (\text{Remark~\ref{remarkbox:jlstar}});
\\
\left[\frac{a}{f^l}\right] & \in  
0^{*I^n}_{\homology^d_\mathfrak{m} (R)}
\qquad (\text{Discussion~\ref{discussionbox:jntightclosure}}).
\qedhere
\end{align*}
\end{remarkbox}

We can now prove the main result of this note.

\begin{proof}[Proof of Theorem~\ref{theorem:test_submodule_charp}]
Since $R$ is complete, $\homology^d_\mathfrak{m} (R)^{\vee} =
\omega_R$ and $\homology^{d+1}_{\mathfrak{M}}
(\overline{\Rees})^{\vee}=\omega_{\overline{\Rees}}$. 
Let $n \geq 1$.
We want to show that
\begin{equation}
\label{equation:tauOmegaReesDegn}
\tau(\omega_R,I^n)
=\left[\tau(\omega_{\overline{\Rees}})\right]_n.
\end{equation}
By Lemma~\ref{lemma:homomorpism_on_local_cohom} 
the natural map (coming from $\psi_{-n}$) 
\[
\frac{\homology^d_\frakm(R)}{0^{*I^n}_{\homology^d_\mathfrak{m} (R)}}
\to
\left[ \frac{\homology^{d+1}_{\mathfrak{M}}
        (\overline{\Rees})}{0^{*}_{\homology^{d+1}_{\mathfrak{M}}
        (\overline{\Rees})}}
\right]_{-n}.
\]
is an isomorphism.
Hence the map
\[
\left[ \left(\frac{\homology^{d+1}_{\mathfrak{M}}
        (\overline{\Rees})}{0^{*}_{\homology^{d+1}_{\mathfrak{M}}
        (\overline{\Rees})}}\right)^{\vee}
\right]_{n}
\to
\left( \frac{\homology^d_\frakm(R)}{0^{*I^n}_{\homology^d_\frakm(R)}}
\right)^{\vee}
\]
is an isomorphism.
This proves~\eqref{equation:tauOmegaReesDegn}.
\end{proof}

We now give some applications of
Theorem~\ref{theorem:test_submodule_charp}.
Prior to that, we describe the graded components 
$\left[\omega_{\overline{\Rees}}\right ]_n$.
Let $X = \Proj (\overline{\Rees} )$, 
$\overline{G} := 
\oplus_{n \in \naturals } \overline{I^n}/\overline{I^{n+1}}$ 
(i.e.  the associated graded ring for the integral closure filtration),
$Y = \Proj \overline{G}$ (the exceptional divisor in $X$)
and $U = X \minus Y$.
Then we can translate~\eqref{equation:locohExactSeq} as below:
\[
0 \to \homology^{d-1}(X, (I\strSh_X )^{-n})
\to \homology^{d-1}(U, \strSh_U)
\to \homology_Y^d(X, (I\strSh_X )^{-n}) \to 0
\]
Taking Matlis duals, we see that
\[
\left[\omega_{\overline{\Rees}}\right ]_n
= \Hom_R(\homology_Y^d(X, (I\strSh_X )^{-n}), E)
= \homology^0(X, I^n \omega_X)
\]
where the last equality follows from `duality with 
supports'~\cite[Theorem, p~.188]{LipmanDesingTwoDim1978}.
(The above paragraph did not use the hypothesis that $\overline{\Rees}$ is
Cohen-Macaulay.)

When $R$ is Gorenstein, we have another description of 
$\omega_{\overline{\Rees}}$ from
\cite{kotal2023blowup} (where this is proved more generally for
$I$-admissible filtrations).
Recall that $\overline{\Rees}$ is Cohen-Macaulay, by assumption.

\begin{proposition}[\protect{\cite[Proposition~2.2.3]{kotal2023blowup}}]
\label{proposition:descrOmegaR}
Let $(R,\mathfrak{m})$ be a Gorenstein local ring, of dimension
$ d\geq 2$.
Let $I$ be an ideal generated by a system of parameters.
Then $\left[\omega_{\overline{\Rees}}\right]_n=I^n:\overline{ I^{d-1}}$
for each $n \geq 1$.
If, further, $\overline{G}$ is Gorenstein, then
$\left[\omega_{\overline{\Rees}}\right]_n=\overline{ I^{n+a+1}}$
where $a$ is the $a$-invariant of $\overline{G}$.
\end{proposition}

We are now ready to prove Corollary~\ref{corollary:dim2}.

\begin{proof}[Proof of Corollary~\protect{\ref{corollary:dim2}}]
Without loss of generality, $I$ is generated by a system of parameters.
First assume that $R$ is $F$-rational. Since it is Gorenstein, it is
$F$-regular. Therefore 
\[
\overline{I^n} \subseteq \tau(I^n) = 
\left[\tau(\omega_{\overline{\Rees}})\right]_n
\subseteq \left[\omega_{\overline{\Rees}}\right]_n = 
\overline{I^n} 
\]
for all $n \geq 1$.
(Use the $F$-regularity of $R$, Theorem~\ref{theorem:test_submodule_charp},
and Proposition~\ref{proposition:descrOmegaR}.)
Hence, by Theorem~\ref{theorem:test_submodule_charp}, $\overline{\Rees}$ is
$F$-rational. (This does not need the hypothesis that $R/\frakm$ is
infinite.)

Now assume that $\Proj \overline{\Rees}$ is $F$-rational.
Hence $\tau(I^n) = \overline{I^n}$ for all $n \gg 1$.
Since $\dim R = 2$, we see that $\overline{I^{n-1}}\tau(I) = 
\overline{I^n}$ for all $n \gg 1$.
Therefore $\tau(I) \subseteq \overline{I}$ is a reduction of 
$\overline{I}$.
Since $R/\frakm$ is infinite, there exists a minimal reduction $I'
\subseteq \tau(I)$, generated by a system of parameters. 
Since $\overline{{I'}^n} = \overline{I^n}$ for all $n$, we may replace $I$
by $I'$ and assume that $I \subseteq \tau(I)$.

From~\cite[Theorem~1.4]{WatanabeYoshidaWangsTheorem2012}, we see that
$I :_R \overline{I} = I + \tau(I)$
On the other hand, since $\overline{G}$ is Gorenstein with $a(\overline{G}
) = -1$, we have that
$I :_R \overline{I} = \left[\omega_{\overline\Rees}\right ]_1 =
\overline{I}$.
Hence $ \tau(I )=  \overline{I}$.
Therefore $ \tau(I^n )=  \overline{I^n} = 
\left[\omega_{\overline\Rees}\right ]_n$ for all $n \geq 1$.
By Theorem~\ref{theorem:test_submodule_charp}, $\overline{\Rees}$ is
$F$-rational.
\end{proof}

\begin{examplebox}
(\cite[Example~3.9(2)]{HaraWatanabeYoshidaFrationality2002})
Consider the ring 
$S = \Bbbk[x,y,z]/(z^2+xy^2+yx^2)$
where $\Bbbk$ is an $F$-finite infinite field of characteristic $2$
and $\frakn = (x,y,z )S$.
Then 
\[
\left[ S[\frakn t]_{xt}\right]_0
\simeq \Bbbk[x,y',z']/(z^2+x'y'(1+y')).
\]
The singular locus of the above ring has two points: $(x,y',z')$ and
$(x,1+y',z')$. When we localize at $(x,y',z')$, we get
\[
\Bbbk[x,y',z']_{(x,y',z')}/(z^2+x'y')
\]
which is $F$-rational. (It is $F$-injective by Fedder's criterion and its
$a$-invariant is negative.) The case $(x,1+y',z')$ is similar.
A similar argument holds for $\left[ S[\frakn t]_{yt}\right]_0$.
Since $(x,y)$ is a reduction for $\frakn$, it follows that $\Proj 
S[\frakn t]$ is $F$-rational.

Now let $(R, \frakm )$ be the completion of $S$ at $\frakn$,
$I = (x,y )R$ and $\Rees = \Rees(\frakm)$. Then $\Rees$ is Cohen-Macaulay
and $\Proj \Rees$ is $F$-rational. In particular $\Rees$ is normal.
Applying Corollary~\ref{corollary:dim2}, we see that 
$\Rees$ is $F$-rational.
\end{examplebox}

\begin{observationbox}
We state a corollary of Theorem~\ref{theorem:test_submodule_charp} that we
believe is currently unknown.
With notation as in Theorem~\ref{theorem:test_submodule_charp},
if there exists $n \geq d$ such that 
\begin{equation}
\label{equation:zeroStarZeroSomen}
\left[{0^{*}_{\homology^{d+1}_{\mathfrak{M}}
        (\overline{\Rees})}}
\right]_{-n} = 0
\end{equation}
then 
\begin{equation}
\label{equation:zeroStarZeroAlln}
\left[{0^{*}_{\homology^{d+1}_{\mathfrak{M}}
        (\overline{\Rees})}}
\right]_{-j} = 0
\end{equation}
for all $j \geq n$.
Indeed,~\eqref{equation:zeroStarZeroSomen} holds if and only if
$\tau(\omega_R, I^n) = \left[\omega_{\overline{\Rees}}\right]_n$.
Since $n \geq d$, it follows that 
\[
\tau(\omega_R, I^{n+1}) 
\subseteq
\left[\omega_{\overline{\Rees}}\right]_{n+1}
=
I \left[\omega_{\overline{\Rees}}\right]_{n}
=
I\tau(\omega_R, I^{n})  
\subseteq
\tau(\omega_R, I^{n+1}) 
\]
from which we obtain~\eqref{equation:zeroStarZeroAlln} with $j = n+1$.
The rest follows by induction.
\end{observationbox}

\section{Characteristic zero}
\label{section:charzero}

In this section, we sketch a proof of Theorem~\ref{theorem:charzero}.

\begin{definition}
Let $(R,\mathfrak{m})$ be an excellent Cohen-Macaulay normal local domain
of dimension $d\geq 2$ 
and essentially of finite type over a field of characteristic $0$.
Let $I$ be an $R$-ideal.
Let $X \to \Spec R$ be a desingularization such that $I \strSh_X$ is
invertible.
Let $n \geq 0$.
Define the \emph{multiplier submodule} $\calJ(\omega_R, I^n)$ to be the
$R$-submodule $\homology^0(X, I^n\omega_X)$ of $\omega_X$.
\end{definition}

In the context of the above definition, suppose that $Y \to \Spec R$ is
another desingularization such that $I \strSh_Y$ is invertible.
We need to show that $\homology^0(Y, I^n\omega_Y) =
\homology^0(X, I^n\omega_X)$.
To this end, we may assume that the map $Y \to \Spec R$ factors through the
map $X \to \Spec R$. Write $f$ for the map $Y \to X$.
Note that
\[
\homology^0(Y, I^n\omega_Y) = \homology^0(X, f_*(I^n\omega_Y))
= \homology^0(X, f_*\omega_Y \otimes I^n\strSh_X)
= \homology^0(X, \omega_X \otimes I^n\strSh_X)
= \homology^0(X, I^n\omega_X).
\]
The second equality is by the projection formula; the next follows since
$X$ is non-singular and, hence, has rational singularities.

\begin{proof}[Proof of Theorem~\ref{theorem:charzero}]
We follow the argument
of~\cite[Remark~2.8]{HyryVillamayorBrianconSkoda1998}.
Write $A = \overline{\Rees(I)}$.
Let $X \to \Spec R$ be a desingularization such that $I \strSh_X$ is
invertible.
Write $X = \Proj \Rees_R(\fraka)$ for some $R$-ideal $\fraka$.
Let $Z = \Proj \Rees_A(\fraka A)$.
Then $Z = \underline{\Spec} \oplus_{n \geq 0} I^n\strSh_X$;
hence the map $Z \to \Spec A$ is a desingularization.
Note that $\omega_Z = \oplus_{n \geq 1} I^{n}\omega_X$.
Thus we get an inclusion
\[
\oplus_{n \geq 1}\calJ(\omega_R, I^n) = \Gamma(Z, \omega_Z) \subseteq
\omega_A.
\]
Moreover, by the Grauert-Riemenschneider vanishing theorem,
$\homology^i(Z, \omega_Z)  = 0$ for all $i \neq 0$.
Let $\frakP \in \Spec A$.
Therefore $A_\frakP$ is a rational singularity if and only if
$\Gamma(Z, \omega_Z)_\frakP = \omega_{A_\frakP}$
if and only if
$\frakP$ is not in the support of $\omega_A/\Gamma(Z, \omega_Z)$.
\end{proof}

\begin{remark}
After seeing the first version of this paper on the arXiv preprint of this paper, Karl Schwede pointed us to 
\newline
\cite{HaconLamarcheSchwedeGlobalGen2021}
and 
\cite{BhattMaPatkfSchwTuckWaldWitsz2023}
where a definition similar to that of 
$\calJ(\omega_R, I^n)$ above has been made and a result similar to
Theorem~\ref{theorem:charzero} has been proved (among other things).
\end{remark}

\ifreadkumminibib
\bibliographystyle{alpha}
\bibliography{kummini}

\else

\def\cfudot#1{\ifmmode\setbox7\hbox{$\accent"5E#1$}\else
	\setbox7\hbox{\accent"5E#1}\penalty 10000\relax\fi\raise 1\ht7
	\hbox{\raise.1ex\hbox to 1\wd7{\hss.\hss}}\penalty 10000 \hskip-1\wd7\penalty
	10000\box7}

\fi %
\end{document}